\documentclass[a4paper,12pt]{amsart}
\usepackage{amssymb}
\usepackage{amsmath}
\usepackage{amscd}
\usepackage{amscd,amsthm,amssymb}

\def\.{\cdot}
\def\la{\langle}
\def\ra{\rangle}

\def\beq{\begin{equation}}
\def\eeq{\end{equation}}
\def\bea{\begin{eqnarray*}}
\def\eea{\end{eqnarray*}}
\def\beaa{\begin{eqnarray}}
\def\eeaa{\end{eqnarray}}
\def\ba{\begin{array}}
\def\ea{\end{array}}

\def\bp{\begin{proof}}
\def\r{\end{proof}}

\def \RM{\mathbb{R}}

\def \HM{\mathbb{H}}
\def \SM{\mathbb{S}}

\def\id{\mathrm{id}}
\def\be{\begin{equation}}
\def\ee{\end{equation}}
\def\tr{\mathrm{tr}}

\def\so{\mathfrak{so}}
\def\su{\mathfrak{su}}
\def\spin{\mathfrak{spin}}
\def\m{\mathfrak{m}}

\def\SU{\mathrm{SU}}

\def\Spin{\mathrm{Spin}}

\def\ad{\mathrm{ad}}

\def\scal{\mathrm{scal}}

\def\vol{\mathrm{vol}}

\newtheorem{epr}{Proposition}[section]

\newtheorem{ath}[epr]{Theorem}

\newtheorem{elem}[epr]{Lemma}

\newtheorem{ecor}[epr]{Corollary}

\theoremstyle{definition}

\newtheorem{ere}[epr]{Remark}
\newtheorem{exe}[epr]{Example}

\author{Andrei Moroianu and Uwe Semmelmann}

\address{Andrei Moroianu \\
Universit{\'e} de Versailles-St Quentin  \\
UMR 8100 du CNRS, 78035 Versailles Cedex\\
45 avenue des {\'E}tats-Unis}
\email{andrei.moroianu@math.cnrs.fr     }

\address{Uwe Semmelmann\\
Institut f\"ur Geometrie und Topologie \\
Fachbereich Mathematik\\
Universit{\"a}t Stuttgart\\
Pfaffenwaldring 57 \\
70569 Stuttgart, Germany
}
\email{uwe.semmelmann@mathematik.uni-stuttgart.de}

\date{\today}
\thanks{This work was initiated during a ``Research in Pairs'' stay at the Mathematisches Forschungsinstitut, Oberwolfach, Germany. We warmly thank the MFO for hospitality. The first author was partially supported by the contract ANR-10-BLAN 0105 ``Aspects Conformes de la G{\'e}om{\'e}trie''. }

\begin{document}

\begin{abstract} 
We study generalized Killing spinors on the standard sphere $\SM^3$, which turn out to be related to Lagrangian embeddings in the nearly K\"ahler manifold $S^3\times S^3$ and to great circle flows on $\SM^3$. Using our methods we generalize a well known result of Gluck and Gu \cite{gg01} concerning divergence-free geodesic vector fields on the sphere and we show that the space of Lagrangian submanifolds of $S^3\times S^3$ has at least three connected components.
\medskip


\end{abstract}

\title{Generalized Killing spinors and Lagrangian graphs}

\maketitle

\section{Introduction}
In this article we investigate  {\it generalized Killing spinors} \cite{bgm} (cf. also \cite{kf01}) on the standard sphere $\SM^3$. Recall that generalized Killing spinors on some spin manifold $(M,g)$ are spinors on $M$ verifying the equation 
\beq\label{def}
 \;  \nabla_X\Psi = A(X)\cdot \Psi   \qquad \forall X\in \mathrm{T}M
\eeq
for some endomorphism  $A$ symmetric with respect to $g$. Real Killing spinors (for $A=\lambda\id$, $\lambda\in\RM$) or parallel spinors (for $A=0$) are particular examples of such objects.
Generalized Killing spinors arise as restrictions of parallel spinors to hypersurfaces, and the converse is true under some analyticity assumption \cite{amm}. This initial problem can thus be understood as an isometric embedding problem for $\SM^3$ into some 4-dimensional hyperk\"ahler ambient space.
Note that in \cite{ms13} we gave examples of genuine (i.e. non-Killing) generalized Killing spinors on $\SM^3$, showing that the problem is non-trivial. 

In our first result we show that generalized Killing spinors on any $3$-dimensional spin manifold $(M,g)$ are in one-to-one correspondence with divergence-free orthonormal frames on $M$. Our examples in \cite{ms13} are equivalent in this setting with frames made by Hopf (left or right-invariant) vector fields for the Killing spinors, and to reflexions of such Hopf left or right-invariant frames with respect to some fixed right or left-invariant Hopf vector field, for the genuine generalized Killing spinors.

We next interpret generalized Killing spinors on $\SM^3$ as maps $f:\SM^3\to \SM^3$ whose differential has the following symmetry property: for every $g\in \SM^3$, the linear map $M_g:\mathrm{T}_e\SM^3\to \mathrm{T}_e\SM^3$ defined by $X\mapsto f_*(gX) f(g)^{-1}$ is symmetric with respect to the standard scalar product on $\mathrm{T}_e\SM^3=\RM^3$. Here $\SM^3$ is viewed as a Lie group with Lie algebra $\mathrm{T}_e\SM^3$, $f(g)^{-1}$ denotes infinitesimal left translation with the inverse of $f(g)$, and $gX$ is the value at $g$ of the left invariant vector field generated by $X$.

This point of view is particularly interesting when considering the graph of $f$ as a submanifold of $S^3\times S^3$ endowed with its 3-symmetric nearly K\"ahler metric. It turns out that the above symmetry property of $f$ is equivalent to the fact that the graph $\Gamma_{f^{-1}}$ of the map $g\mapsto f(g)^{-1}$ is a Lagrangian submanifold of 
the nearly K\"ahler $S^3\times S^3$, endowed with its fundamental two-form $\Omega$. Note that the terminology is somewhat improper since $\Omega$ is not closed. Nonetheless, Lagrangian submanifolds of 6-dimensional nearly K\"ahler manifolds were intensively studied in the last decades, perhaps motivated by the fact that they are automatically minimal (cf. Ejiri \cite{e81} in the case of $\SM^6$ and Gutowski-Ivanov-Papadopoulos \cite{gip03} in general).

Until now, the only known examples (up to isometry) of Lagrangian submanifolds of 
the nearly K\"ahler $S^3\times S^3$ were the factors and the diagonal. Our examples of genuine generalized Killing spinors on $\SM^3$ yield in this way new examples of Lagrangian graphs of $S^3\times S^3$, but we have also found an interesting family of Lagrangian submanifolds of $S^3\times S^3$ which project onto a strict submanifold on each factor. We computed the metric structure for each of the examples, which, for the usual normalization of the metric on $S^3\times S^3$, turn out to be round spheres of radius $\frac23$ and $\frac43$, as well as Berger spheres of some different volume. This last observation shows that the space of Lagrangian submanifolds of $S^3\times S^3$ has at least three connected components.

We have also investigated generalized Killing spinors on $\SM^3$ by comparing them to some fixed Killing spinor. In this way, every generalized Killing spinor on $\SM^3$ is characterized by a function $\alpha$ and a vector field $\xi$ satisfying some coupled non-linear differential system. In the particular case where the function $\alpha$ vanishes, the system reduces to the condition that $\xi$ is a geodesic divergence-free vector field. Such objects were studied by Gluck and Gu \cite{gg01}, who showed (using a nice interpretation as holomorphic graphs in the oriented Grassmannian $\tilde{\mathrm{Gr}}_2(\RM^4)$) that they are necessarily Hopf vector fields. Translating back into our setting, we obtain as a corollary that every generalized Killing spinor on $\SM^3$ whose scalar product with some Killing spinor vanishes is necessarily in the list of our known examples. We finally generalize this result to the case when this scalar product is constant but not necessarily zero, by solving an ODE along the orbits of $\xi$. 

\bigskip

\section{Spinors on 3-manifolds and divergence-free frames}

Let $(M^3,g)$ be a 3-dimensional spin manifold. 
Since the spin representation $\Spin(3)\to\mathrm{Aut}(\Sigma_3)$ is isomorphic to the left multiplication of unit quaternions on $\Sigma_3\simeq\HM$, the spinor bundle $\Sigma M$ has a quaternionic structure, acting from the right. This structure is compatible with the natural scalar product $\la\cdot,\cdot\ra$  on the spin bundle, in the sense that $\la\Psi a,\Phi a\ra=|a|^2 \la\Psi ,\Phi \ra$ for every $a\in\HM$ and $\Psi,\Phi\in\Sigma M$.

In particular, if $a\in \mathrm{Im \HM}$ is an imaginary quaternion, then $a^2=-|a|^2$, so for every spinor $\Psi$ we have
$$|a|^2\la\Psi a,\Psi\ra=\la\Psi a^2,\Psi a\ra=-|a|^2\la \Psi,\Psi a\ra,$$
showing that $\Psi a$  is orthogonal  to $\Psi$. On the other hand, for every $x\in M$ and non-zero $\Psi\in\Sigma_xM$, the map $\mathrm{T}_xM\to \Psi^\perp\subset \Sigma_xM$, mapping $X$ to $X\cdot\Psi$ is an isomorphism (for dimensional reasons). For every imaginary quaternion $a$ and nowhere vanishing spinor field $\Psi\in C^\infty (\Sigma M)$ we can thus define a vector field $\xi_a$ on $M$ by
\be\label{xia}\xi_a \cdot \Psi = \Psi a.\ee

We now give a characterization of generalized Killing spinors in terms of the associated unit vector fields $\xi_a$. Recall that every generalized Killing spinor $\Psi$ has constant length. Indeed, from \eqref{def} we get $X(|\Psi|^2)=2\la\nabla_X \Psi,\Psi\ra=2\la A(X)\cdot\Psi,\Psi\ra=0$ for every vector field $X$.

\begin{elem} \label{div}
A spinor $\Psi\in C^{\infty}(\Sigma M)$ of constant length is a generalized Killing spinor if and only if the vector fields $\xi_a$ are divergence-free for all $a\in \mathrm{Im \HM}$. 
\end{elem}
\begin{proof}

Since $\Psi$ has constant length, $\nabla_X\Psi$ is orthogonal to $\Psi$ at every point, so by the linearity of the covariant derivative, there exists some endomorphism $A$ of $\mathrm{T}M$ such that 
\be\label{gks}\nabla_X\Psi = A(X) \cdot \Psi\ee 
for any vector field $X$. 
Taking the covariant derivative in the defining equation \eqref{xia} yields
$$
\nabla_X \xi_a \cdot \Psi +  \xi_a \cdot A(X) \cdot \Psi  =\nabla_X(\Psi a)=(\nabla_X\Psi) a= A(X) \cdot \Psi  a 
=  A(X) \cdot \xi_a \cdot \Psi \ .
$$
Hence
$$
\nabla_X \xi_a \cdot \Psi = 2A(X) \wedge \xi_a \cdot \Psi = -2 \ast ( A(X) \wedge \xi_a) \cdot \Psi \ ,
$$
and it follows
\be\label{nxi}
\nabla_X \xi_a = -2 \ast (A(X) \wedge \xi_a) = - 2A(X) \lrcorner \ast \xi_a \ .
\ee
Assume now that $\Psi$ is a generalized Killing spinor, i.e. that the endomorphism $A$ defined in Equation \eqref{gks} is symmetric. We obtain
$$
\delta \xi_a =  -  e_i \lrcorner \nabla_{e_i}\xi_a =  2e_i \lrcorner A(e_i) \lrcorner \ast \xi_a = 0 \ .
$$
(Here and in the following we use Einstein's summation convention over repeated subscripts).

Conversely, if $\xi_a$ are divergence-free, \eqref{nxi} shows that  $e_i \lrcorner A(e_i) \lrcorner \ast \xi_a = 0$  for every  $a\in \mathrm{Im \HM}$, i.e. the 2-form $e_i\wedge A(e_i)$ vanishes. Since this two-form represents the skew-symmetric part of $A$, the lemma follows.

\end{proof}

Since every oriented orthonormal frame labeled $(\xi_i,\xi_j,\xi_k)$ defines (up to sign) a unique spinor of unit length satisfying \eqref{xia} for $a=i,j,k$, the previous lemma gives at once:

\begin{ecor} Generalized Killing spinors on $M$ are (up to sign) in 1-1 correspondence with oriented orthonormal frames of divergence-free vector fields on $M$.
\end{ecor}

\section{Spinors on $\SM^3$}

In this section we describe spinors on the round sphere $\SM^3$ and  translate the
generalized Killing equation into conditions on the definig function in a left-invariant frame.

We consider $\SM^3$ as the unit sphere in $\HM$, with the induced Lie group structure. In this way $\SM^3$ is identified to $\SU(2)$ and the Lie algebra of $\SM^3$ is identified with $\mathrm{Im \HM}\simeq \su(2)$. More generally, the tangent space $\mathrm{T}_g\SM ^3$ is identified to $g\mathrm{Im \HM}$ and the infinitesimal left and right translations on $\SM^3$ are given by left or right quaternionic products.
Let $(e_1, e_2, e_3)=(i,j,k)$ be a fixed basis in $\mathrm{T}_e\SM^3 = \mathrm{Im \HM}$ (positively oriented by convention). 
Then left translation  defines an orthonormal frame on $\SM^3$:
$u(g) := (ge_1, ge_2, ge_3)$.  We will take $u$ as a reference frame and we endow $\SM^3$ with the orientation induced by $u$.
The Levi-Civita connection on left-invariant vector fields $X_g=gx$ and $Y_g=gy$ is
given by the well known formula 
\beq\label{xy}(\nabla_{X}Y)_g= \tfrac12 g [x,y],
\eeq 
for any $g\in \SU(2) $
and $x,y \in \su(2)$. More generally, if $Y$ is any vector field on $\SM^3$, we can write $Y_g=gy(g)$ for every $g\in \SM^3$ where $y:\SM^3\to \mathrm{T}_e\SM^3$ is some smooth function, and we have
\be\label{nabla}(\nabla_{X}Y)_g=g\left(\tfrac12 [x,y(g)]+y_*(X)\right).
\ee

With respect to the fixed frame $u$ the  connection $1$-form
$\omega$ of the Levi-Civita connection is  
$$
\omega(X) = \tfrac12 \ad_x = \ast x \in \so(3)\cong \RM^3 \ ,
$$
for every tangent vector $X\in \mathrm{T}_g\SM^3$ written as $X=gx$, $x\in \mathrm{T}_e\SM^3\cong \RM^3$.
Here and henceforth we identify vectors and 1-forms using the Riemannian metric.
We denote by $\tilde u$ a lift of the frame $u$ to a section of the spin principal bundle. Any
spinor $\Psi$ can then be written as
\beq\label{psif}
\Psi=[\tilde u, f],
\eeq
for some function $f$ defined on $\SM^3$ with values in the spin module, which in our case can be identified with $\HM$. Since $ijk=-1$  it follows that $X\cdot Y\cdot Z\cdot\Psi=-\Psi$ for every positive orthonormal base $X,Y,Z$. This shows that the Clifford action of $X$ and $\ast X$ are related by 
\beq\label{ast} \ast X \cdot \Psi = X \cdot \Psi\eeq
for every tangent vector $X$ and for every spinor $\Psi$ (note that in \cite{ms14} the opposite sign convention was used).

The covariant derivative of $\Psi$ with respect to some tangent vector $X=gx=[u(g),x]$ is given by
\be \label{cov}
\nabla_{X}\Psi = [\tilde u,X(f)+\widetilde{\omega(X)} \cdot f] \ .
\ee
Here $A\mapsto \tilde A$ denotes the inverse of the differential of the spin covering, which via the isomorphism $\so(3)\cong\spin(3)$ corresponds to the multiplication by $\frac12$. We infer 
\beq\label{cov1}
[\tilde u, \widetilde{\omega(X)} \cdot f] =[\tilde u, \widetilde{\ast x} \cdot f  ] = \tfrac12 \ast X \cdot \Psi = \tfrac12 X \cdot \Psi\ .
\eeq
In the particular case of a constant function $f$ we thus obtain from \eqref{cov} and \eqref{cov1}:
\be \label{ks}
\nabla_{X}\Psi = [\tilde u, \widetilde{\omega(X)} \cdot f]=\tfrac12 X \cdot \Psi \ .
\ee
Hence any spinor given by a constant function with respect to a left-invariant frame is a
Killing spinor for the Killing constant $\frac12$. Similarly, constant spinors with respect to a right-invariant frame are Killing spinors with
Killing constant $-\frac12$.

Consider the unit vectors $\xi_a$ defined by $\Psi$ via \eqref{xia}.
Our next goal is to interpret the condition $\delta \xi_a = 0$ from Lemma \ref{div} in terms of the function $f$ defining the spinor $\Psi$ in the left-invariant frame $\tilde u$ (cf. Equation \eqref{psif}). 

Writing  $\xi_a(g)= [u, v_a]= g v_a(g)$ for some function $v_a : \SM^3 \rightarrow \mathrm{Im \HM} \cong \RM^3$,
Equation \eqref{xia}  translates into 
\be\label{va}
v_a(g) =  f(g) a f(g)^{-1 }\qquad \hbox{or, equivalently,}\qquad \xi_a(g) =  g f(g) a f(g)^{-1 }  \ .
\ee

\begin{elem} \label{symmetric}
Let $\Psi=[\tilde u,f]$ be a spinor on $\SM^3$ with associated vector fields $\xi_a$ defined by \eqref{xia}. Then the vector fields $\xi_a$ are divergence-free for all $a\in \mathrm{Im \HM}$ if and only if for every $g\in \SM^3$, the endomorphism $M_g$ of $\mathrm{Im \HM}$ defined by $M_g(x):= f_\ast (gx)f^{-1}(g) $ is symmetric. 
\end{elem}
\begin{proof}
We compute the covariant derivative of $\xi_a$ in the
direction of some vector $X_g=gx$ (with $x\in\mathrm{Im \HM}$) and obtain from \eqref{nabla} and \eqref{va}
$$
(\nabla_{X} \xi_a)_g = \frac12 g\left( [x, f(g) a f^{-1}(g)] + f_\ast (gX)af^{-1}(g)
-f(g)af^{-1}(g) f_\ast (gX)f^{-1}(g)  \right) \ .
$$
Using this formula we may calculate the divergence of the vector field $\xi_a$, where 
we set $b:= f(g) a f^{-1}(g)$:
\bea
-(\delta \xi_a)_g &=& \tfrac12 \la [e_i, f(g)af^{-1}(g) ], e_i \ra+ \la f_\ast (ge_i) f^{-1}(g)b , e_i\ra - \la b f_\ast (ge_i)  f^{-1}(g), e_i\ra
\\[1.5ex]
&=& \la M_g(e_i), -e_i b + b e_i\ra = \la M_g(e_i), [b,e_i] \ra = \la[e_i, M_g(e_i)], b \ra \\[1.5ex]
&=& 2\la \mathrm{vol}, e_i \wedge M_g(e_i) \wedge b\ra\ .
\eea
For any fixed $g$, when $a$ runs through $\mathrm{Im \HM}$, $b$ takes any value in $\mathrm{Im \HM}$. It follows
that $\delta \xi_a = 0$ for all $a$ if and only if $e_i \wedge M_g(e_i) \wedge b = 0$ for
all $b$, which is equivalent to $e_i \wedge M_g(e_i) = 0$ and finally to $M_g$ being symmetric.

\end{proof}

\begin{exe} \label{egks} In \cite{ms14} we have constructed the following examples of generalized Killing spinors on $\SM^3$:
\begin{enumerate}
\item Killing spinors with constant $\frac12$.
\item Killing spinors with constant $-\frac12$.
\item Products $\xi\cdot\Phi$ where $\xi$ is right-invariant and $\Phi$ is a Killing spinor with constant $\frac12$.
\item Products $\xi\cdot\Phi$ where $\xi$ is left-invariant and $\Phi$ is a Killing spinor with constant $-\frac12$.
\end{enumerate}
\end{exe}
\begin{elem} \label{ex}
Up to left multiplication of $f$ with some constant quaternion, the above examples of generalized Killing spinors
correspond to the following functions and vector fields. 
$$
\begin{array}{lll}
& (1) \quad f(g)   =  1                                  \qquad    &  \xi_a(g) = g a \\
& (2) \quad f(g)   = g^{-1}                                   \qquad     &  \xi_a(g)  = a g \\
& (3) \quad f(g)   = g^{-1}bg                        \qquad    &  \xi_a(g)  = b gag^{-1}b^{-1}g \\
& (4) \quad f(g)   = bg^{-1}                                 \qquad    &  \xi_a(g)  = gbg^{-1}agb^{-1}
\end{array}
$$
\end{elem}
\begin{proof}
Consider a left-invariant frame $u$ and its spin lift $\tilde u$ as before. We have already seen in Equation \eqref{ks} that Killing spinors with constant $\frac12$ correspond to constant functions $f$ in the frame $\tilde u$. Consider the spinor $\Psi:=[\tilde u, f]$ with $f(g)=g^{-1}$. For every vector $X\in \mathrm{T}_g\SM^3$, written as $X=gx=[u,x]$, the derivative of the $\HM$-valued function $f$ with respect to $X$ reads $X(f)=-xg^{-1}=-x\cdot f$.
Using Equation \eqref{cov} we can thus compute the covariant derivative of $\Psi$:
$$\nabla_{X}\Psi = [\tilde u,X(f)+\widetilde{\omega(X)} \cdot f]=[\tilde u,X(f)]+\tfrac12X\cdot\Psi=-[u,x]\cdot[\tilde u,f]+\tfrac12X\cdot\Psi=-\tfrac12X\cdot\Psi.$$
This proves case (2). 

If $\xi$ is right-invariant and $\Phi$ is a Killing spinor with constant $\frac12$ we can write $\xi_g=bg=[u,g^{-1}bg]$ for some $b\in  \mathrm{Im \HM}$ and (up to right multiplication with a constant) $\Phi=[\tilde u,1]$, whence $\xi\cdot\Phi=[\tilde u, g^{-1}bg]$. Similarly, if $\xi$ is left-invariant and $\Phi$ is a Killing spinor with constant $-\frac12$ we can write $\xi_g=gb=[u,b]$ for some $b\in  \mathrm{Im \HM}$ and (up to right multiplication with a constant) $\Psi=[\tilde u,g^{-1}]$, whence $\xi\cdot\Phi=[\tilde u, bg^{-1}]$. The corresponding formulas for $\xi_a$ follow from Equation \eqref{va}.

\end{proof}

\section{Lagrangian graphs}

The graph $\Gamma_f$ of a smooth map $f: \SM^3 \rightarrow \SM^3$ defines a submanifold in
$S^3 \times S^3$. In this section we want to show that the symmetry condition in Lemma \ref{symmetric} above
translates into the the fact that the graph $\Gamma_{f^{-1}}$ of the map $g\mapsto f(g)^{-1}$ is Lagrangian with respect to a certain non-degenerate 2-form on $S^3 \times S^3$.

We identify $S^3 \times S^3$ with the homogeneous space $S^3 \times S^3 \times S^3/\Delta(S^3)$
via the action 
$$(g_1, g_2, g_3) \cdot (a_1, a_2 ) := (g_1a_1g^{-1}_3, g_2a_2g^{-1}_3).$$ 
The stabilizer of $(e,e)$ is then the diagonal of $S^3 \times S^3 \times S^3$ and the projection 
$\pi : S^3\times S^3 \times S^3 \rightarrow S^3 \times S^3$ is given by
$\pi(g_1, g_2, g_3) = (g_1g_3^{-1}, g_2g_3^{-1})$.

The tangent space at $(e,e)$ is identified with 
$$\m:=\{(X_1, X_2, X_3) | X_i \in \su(2), X_1 + X_2 + X_3 = 0 \}.$$
In this identification, a tangent 
vector $(Y_1, Y_2)$ corresponds to
\beq\label{id}
\pi^\m(Y_1, Y_2, 0)  - \tfrac13 (Y_1 + Y_2,Y_1 + Y_2,Y_1 + Y_2 )
= \tfrac13(2Y_1 - Y_2, 2Y_2 -Y_1, -Y_1 - Y_2) \ .
\eeq

Let $B$ be the Killing form of $\su_2$, and denote by $B_0 := \frac{1}{12}B$ its rescaling. Then 
$$
g((X_1, X_2, X_3), (Y_1, Y_2, Y_3)) :=- (B_0(X_1, Y_1) + B_0(X_2,Y_2) + B_0(X_3,Y_3))
$$
defines a homogeneous nearly K\"ahler metric of scalar curvature $\scal = 30$ on $S^3\times S^3$ (cf. \cite{ms10}, Lemma 5.4). Denoting $-B_0$ simply by $\la\cdot,\cdot \ra$, and using the identification \eqref{id}, the induced metric on $S^3\times S^3$ reads
\beq\label{g}
g((X_1, X_2), (Y_1, Y_2)) =
\tfrac13( 2 \la X_1, Y_1\ra + 2\la X_2, Y_2 \ra - \la X_1, Y_2\ra - \la X_2, Y_1\ra ) \ .
\eeq

The manifold $S^3 \times S^3$ has the structure of  a $3$-symmetric space. The corresponding
almost complex structure is defined as
$$
J(X_1, X_2, X_3) = \tfrac{2}{\sqrt{3}} (X_3, X_1, X_2) + \tfrac{1}{\sqrt{3}} (X_1, X_2, X_3) \ ,
$$
which by \eqref {id} can be rewritten as
$$
J(X_1, X_2) = \tfrac{1}{\sqrt{3}} (X_1 - 2 X_2, 2X_1 - X_2)\ .
$$
Let $\Omega$ be the fundamental $2$-form $\Omega (A,B) = g(JA, B)$, then
$$
\Omega((X_1, X_2), (Y_1, Y_2))
= \tfrac{1}{\sqrt{3}} (\la X_1,  Y_2 \ra - \la X_2, Y_1\ra) \ .
$$
At an arbitrary point 
$(g_1, g_2) \in S^3 \times S^3$, the $2$-form $\Omega$ is defined by
\beq\label{omega}
\Omega((X_1, X_2),(Y_1,Y_2)) := \tfrac{1}{\sqrt{3}} (\la g_1^{-1}X_1, g_2^{-1}Y_2    \ra
- \la g_2^{-1}X_2, g_1^{-1} Y_1 \ra) \ .
\eeq

A $3$-dimensional submanifold $L$ of the nearly K\"ahler manifold $S^3 \times S^3$
is called {\it Lagrangian} if $\Omega(A,B) = 0$ for all tangent vectors $A, B \in \mathrm{T}L$. Notice
that this is a generalization of the usual concept of the a Lagrangian submanifold since
 the fundamental $2$-form $\Omega$ is not closed.

\begin{elem}\label{lag}
Let $f: \SM^3 \rightarrow \SM^3$ be a smooth map. Then the endomorphism $M_g$ defined
in Lemma \ref{symmetric} is symmetric for all $g$
if and only if the graph $\Gamma_{f^{-1}}$ of $f^{-1}$ is a Lagrangian submanifold
of $S^3 \times S^3$ with respect to the $2$-form $\Omega$.
\end{elem}
\proof
The tangent space to the graph $\Gamma_{f^{-1}}$ at $ (g, f(g)^{-1})$ is the set of vectors of the form
$(gx,-f(g)^{-1}f_\ast(gx)f(g)^{-1})$ for $x \in \su(2)$. 

By \eqref{omega}, the value of $\Omega$
at the point $(g,f(g)) $ on two such tangent vectors is 
$$
-\tfrac{1}{\sqrt{3}} (\la x ,f_\ast(gy) f(g)^{-1}  \ra  - \la f_\ast(gx)f(g)^{-1}  , y   \ra) =-\tfrac{1}{\sqrt{3}} (\la x ,M_g (y)\ra -\la M_g(x),y\ra)\ ,
$$
which shows that $\Gamma_{f^{-1}}$ is Lagrangian with respect to $\Omega$ if and only if $M_g$ is symmetric for all $g\in \SM^3$.

\qed

%

Let $L \subset S^3 \times S^3$ be a $3$-dimensional Lagrangian submanifold. Then the
isometries of $S^3\times S^3$ define deformations of $L$. Indeed the isometry group
of the $3$-symmetric space $S^3 \times S^3$ is $9$-dimensional, whereas the isometry
group of $L$ is at most $6$-dimensional, obtained in the case of the round $3$-sphere.
Hence there remains an at least $3$-dimensional space of normal deformations.

In \cite{ss10} (see also \cite{gip03}) it was shown that for every infinitesimal deformation transversal to the diffeomorphisms of $L$, the so called variation $1$-form is a co-closed eigenform of the Hodge-Laplace operator of $L$ for the eigenvalue $9$, when the scalar curvature of the nearly K\"ahler manifold $S^3 \times S^3$ is normalized to $30$.

\begin{epr}
\begin{enumerate}
\item The submanifolds
$\Gamma_1:= \{ (g, 1) | g \in \SM^3 \} $ and $\Gamma_2: = \{ (g, g) | g \in \SM^3 \} $ 
are Lagrangian submanifolds of $S^3\times S^3$ isometric to  the round sphere  $\SM^3(\frac23)$, of volume $\frac{8}{27}\vol(\SM^3)$. 
\item For every $b\in \SM^2\subset \SM^3$, the submanifolds
 and 
$\Gamma_3(b):= \{ (g, g^{-1}b g) | g \in \SM^3 \} $ and $\Gamma_4(b): = \{ (g, gb) | g \in \SM^3 \} $ are Lagrangian submanifolds of $S^3\times S^3$ isometric to  the Berger sphere obtained from 
$\SM^3(\frac{2}{\sqrt{3}})$ by rescaling the metric on the fibres of the Hopf fibration by a factor $\frac{1}{\sqrt{3}}$. Their volume is equal to $\frac{24}{27}\vol(\SM^3)$.
\item
For every $a, b \in \SM^2 \subset \SM^3, a \perp b$, the submanifolds $L(a,b):=\{gag^{-1}, gbg^{-1}) | g \in \SM^3  \}$ are are Lagrangian submanifolds of $S^3\times S^3$ isometric to  the round sphere  $\SM^3(\frac43)$, 
of volume $\frac{64}{27}\vol(\SM^3)$. \
\end{enumerate}
\end{epr}

\proof The metric structure of the above submanifolds is an immediate consequence of \eqref{g}. The Lagrangian property follows from Lemmas \ref{ex} and \ref{lag} in the first two cases, as $\Gamma_1$, $\Gamma_2$, $\Gamma_3(b)$ and $\Gamma_4(-b)$ are the graphs $\Gamma_{f^{-1}}$ for $f$ in one of the cases (1)--(4) of Lemma \ref{ex} respectively. The verification of the fact that $L(a,b)$ are Lagrangian is straightforward using \eqref{omega}.

\qed

It is well known that the Laplace spectrum on co-closed $1$-forms on the round
sphere $\SM^3(r)$ is given by  $\{  \frac{k^2}{r^2} | k = 2,3, \ldots  \}$. Note that the
first eigenspace, corresponding to the eigenvalue $\frac{4}{r^2} $, is exactly
the space of (1-forms dual to) Killing vector fields.

\begin{ecor}
The radius of a round Lagrangian sphere in $S^3 \times S^3$ is necessarily of the form $\frac{k}{3}$
for some integer $k\ge2$. The values $k=2$ and $k=4$ are realized in our
examples.
\end{ecor}

\bigskip

\begin{ere}
The volume is constant on connected components of the space of Lagrangian
submanifolds. Hence, in the case of $S^3 \times S^3$ it has at least
$3$ connected components.
\end{ere}

\section{Geodesic vector fields}

Let us fix throughout this section the unit length spinor $\Phi:=[\tilde u,1]$. By \eqref{ks}, $\Phi$ is a Killing spinor with Killing constant $ \frac12$. Compared to $\Phi$, any spinor $\Psi$ on $\SM^3$ is determined by a vector field $V$ and a function $\alpha$. Indeed the map $\mathrm{T}_g\SM^3\times \RM\to\Sigma_g \SM^3$ defined by $(V,\alpha)\mapsto V\cdot\Phi+\alpha\Phi$ is bijective at every point $g$, so the spinor $\Psi$ can be uniquely written as
\beq\label{valpha}
\Psi = V\cdot\Phi+\alpha\Phi\ .
\eeq

The generalized Killing equation for $\Psi$ translates into a  system of equations for $V$ and $\alpha$:

\begin{epr}\label{system} The spinor $\Psi := V\cdot\Phi+\alpha\Phi$ is a generalized Killing spinor 
of unit length if and only if the following system holds:
$$
\begin{array}{ll}
& (\mathrm{i}) \quad \alpha^2 + |V|^2 = 1 \\
& (\mathrm{ii}) \quad -V \lrcorner \ast \nabla_VV - V(\alpha) V +\alpha \nabla_VV + d\alpha = 0 \\
& (\mathrm{iii}) \quad \alpha  \ast(V \wedge dV) + (2\alpha -\delta V)(1-\alpha^2) + \alpha V(\alpha) = 0\\
\end{array}
$$
\end{epr}
\proof We assume that $\Psi$ has unit length, which is equivalent to (i). It remains to show that the generalized Killing condition is equivalent to (ii)-(iii). We compute the covariant derivative of $\Psi$ using Equation \eqref{ast}:
\bea\nabla_X\Psi&=&\nabla_XV\cdot\Phi+X(\alpha)\Phi+\tfrac12 V\cdot X\cdot\Phi+\tfrac12 \alpha X\cdot\Phi\\
&=&\left(\nabla_XV+\tfrac12\ast(V\wedge X)+\tfrac12\alpha X\right)\cdot\Phi+\left(X(\alpha)-\tfrac12\la V,X\ra\right)\Phi.\eea
On the other hand, for every vector field $Y$ we have
$$ Y\cdot\Psi=Y\cdot V\cdot\Phi+\alpha Y\cdot\Phi=\left(\ast(Y\wedge V)+\alpha Y\right)\cdot\Phi-\la Y,V\ra\Phi.$$
Consequently, the tensor $A$ defined in Equation \eqref{gks} satisfies
\bea\la A(X),Y\ra&=&\la \nabla_X \Psi, Y\cdot\Psi\ra\\
&=&\la \nabla_XV+\tfrac12\ast(V\wedge X)+\tfrac12\alpha X, \ast(Y\wedge V)+\alpha Y\ra-\la Y,V\ra\left(X(\alpha)-\tfrac12\la V,X\ra\right)
\eea
Using the standard properties of the Hodge adjoint $\ast$ we readily obtain
\bea A(X)&=&-V\lrcorner \ast \nabla_XV +\alpha \nabla_XV-\tfrac12 |V|^2 X+\tfrac12 \la X,V\ra V+\tfrac12\alpha\ast(V\wedge X)\\
&&-\tfrac12 \alpha V\lrcorner \ast X+\tfrac12\alpha^2 X-VX(\alpha)+\tfrac12 \la X,V\ra V\\
&=&-V\lrcorner \ast \nabla_XV+\alpha \nabla_XV -\tfrac12 |V|^2 X+ \la X,V\ra V- \alpha V\lrcorner \ast X+\tfrac12\alpha^2 X-VX(\alpha).
\eea
It was already noticed that $A$ is symmetric if and only if $e_i\wedge A(e_i)=0$ for some local orthonormal basis $e_i$. From the previous formula we get:\bea e_i\wedge A(e_i)&=& -e_i\wedge (V\lrcorner \ast \nabla_{e_i}V)+\alpha dV-\alpha e_i\wedge (V\lrcorner \ast e_i)+V\wedge d\alpha\\
&=&V\lrcorner(e_i\wedge \ast \nabla_{e_i}V)-\ast\nabla_VV+\alpha dV+ 2\alpha\ast V+V\wedge d\alpha\\
&=&-\delta V\ast V -\ast\nabla_VV+\alpha dV+2\alpha\ast V+V\wedge d\alpha=:\omega.
\eea
The 2-form $\omega$ vanishes identically if and only if its wedge and interior product with $V$ vanish. This is clear on the support of $V$, and $\omega$ is zero anyway outside the support of $V$. We now compute:
\bea V\wedge\omega&=&(2\alpha-\delta V)\ast|V|^2-V\wedge\ast\nabla_VV+\alpha V\wedge dV\\
&=&(2\alpha-\delta V)\ast|V|^2-\ast\la V,\nabla_VV\ra+\alpha V\wedge dV\\
&=&\ast\left((2\alpha-\delta V)(1-\alpha^2)+\alpha V(\alpha)\right)+\alpha V\wedge dV,
\eea
and
\bea V\lrcorner \omega&=&-V\lrcorner \ast\nabla_VV+\alpha V\lrcorner dV+|V|^2d\alpha-V(\alpha)V\\
&=&-V\lrcorner \ast\nabla_VV+\alpha \nabla_VV-\tfrac12\alpha d(|V|^2)+(1-\alpha^2)d\alpha-V(\alpha)V\\
&=&-V\lrcorner \ast\nabla_VV+\alpha \nabla_VV+\alpha^2 d\alpha+(1-\alpha^2)d\alpha-V(\alpha)V\\
&=&-V\lrcorner \ast\nabla_VV+\alpha \nabla_VV+d\alpha-V(\alpha)V.
\eea
This proves that the symmetry of $A$ is equivalent to (ii)-(iii).

\qed

\bigskip

The main result in this section is the following 

\begin{ath}\label{theo}
If the function $\alpha$ associated via \eqref{valpha} to a generalized Killing spinor $\Psi$ on $\SM^3$ is
constant, then $\Psi$ is one of the spinors described in cases $(1)$ and $(3)$ in Example \ref{egks} above.
\end{ath}
\proof

Assume first that the function $\alpha$ is identically zero. Then Lemma~\ref{system} implies
that $V$ is a unit length divergence-free vector field satisfying $V \lrcorner \ast \nabla_VV = 0$. As $\la V,\nabla_VV\ra=0$, 
it follows that $\nabla_VV=0$. Using a result of Gluck and Gu (\cite{gg01}, Theorem A) we conclude
that $V$ has to be a Hopf vector field, i.e. a left or right-invariant unit vector field on $\SM^3$. If $V$ is left-invariant then the function representing $\Psi=V\cdot\Phi$ in the frame $\tilde u$ is constant, so $\Psi$ is a Killing spinor with Killing constant $\frac12$. If $V$ is right-invariant, then we are in case (3) of Example \ref{egks}.

If $\alpha$ is identically 1, then $V=0$ so $\Psi=\Phi$ and we are in case $(1)$ of Example \ref{egks}.

Finally, in the case where $\alpha$ is a constant different from 0 and 1,
Lemma~\ref{system} (ii) reads $V \lrcorner \ast \nabla_VV=\alpha \nabla_VV$ and since these two vectors are orthogonal, they both vanish, showing that $V$ is a geodesic vector field.  Using Lemma~\ref{system} (iii) we obtain that the normalized 
vector field $\xi := \frac{V}{|V|}$ satisfies the equation
\beq\label{2}
\ast(\xi \wedge d\xi)+ 2 - \delta \xi \tfrac{\sqrt{1-\alpha^2}}{\alpha} = 0 \ .
\eeq
Let us write $\nabla\xi=\phi+\psi$ with $\phi$ symmetric and $\psi$ skew-symmetric. As $\nabla_\xi\xi=0$, both $\phi$ and $\psi$ vanish on $\xi$. In dimension 2 every trace-free symmetric endomorphism anti-commutes with every skew-symmetric endomorphism, consequently the trace-free part $\phi_0$ of $\phi$ anti-commutes with $\psi$. Writing $\phi=\phi_0+\frac12(\tr\phi)\id$ we infer $(\phi+\psi)^2=\phi^2+\psi^2+(\tr\phi)\psi$, so the skew-symmetric part of $(\phi+\psi)^2$ equals $(\tr\phi)\psi=-(\delta\xi)\psi$. This can be written as follows:
\beq\label{dxi} e_i\wedge\nabla_{\nabla_{e_i}\xi}\xi=e_i\wedge (\phi+\psi)^2(e_i)=-2(\delta\xi)\psi=-(\delta\xi)d\xi,
\eeq
where $e_i$ is any local orthonormal frame.
Using \eqref{dxi} and the fact that the sectional curvature of $\SM^3$ is 1, we compute in a local orthonormal frame $e_i$ parallel at some point:
\bea \nabla_\xi d\xi&=&\nabla_\xi\left(e_i\wedge\nabla_{e_i}\xi\right)=e_i\wedge\left(R_{\xi,e_i}\xi+\nabla_{e_i}\nabla_\xi\xi+\nabla_{[\xi,e_i]}\xi\right)\\
&=&e_i\wedge\left(e_i-\la e_i,\xi\ra\xi-\nabla_{\nabla_{e_i}\xi}\xi\right)=-e_i\wedge\nabla_{\nabla_{e_i}\xi}\xi=(\delta\xi)d\xi,
\eea
thus $\nabla_\xi (\xi \wedge d\xi)= (\xi \wedge d\xi) \delta \xi$
and from \eqref{2} we get $\nabla_\xi \delta \xi = (\delta \xi)^2$. Every orbit of the flow of $\xi$ is a great circle, so is closed. The restriction of $\delta\xi$ to such an orbit is thus a periodic solution of the equation $y'=y^2$. The only periodic solution of this equation being the zero function, we obtain that $\delta\xi$ vanishes on $\SM^3$. By Theorem A in \cite{gg01} again, $\xi$ is a Hopf vector field on $\SM^3$. Moreover, \eqref{2} gives $\ast(\xi \wedge d\xi)=-2$. It is easy to check from \eqref{xy} that $d\xi=-2\ast\xi$ when $\xi$ is left-invariant and $d\xi=2\ast\xi$ when $\xi$ is right-invariant. Consequently $\xi$ is a left-invariant vector field and finally $\Psi=\alpha \Phi+\sqrt{1-\alpha^2}\, \xi\cdot\Phi$ is a Killing spinor with Killing constant $\frac12$. 

\qed

Coming back to our description of generalized Killing spinors on $\SM^3$ in terms of
Lagrangian embeddings, we obtain the following:

\begin{ecor}
Let $f: \SM^3 \rightarrow \SM^2 \subset \SM^3$ be a map whose graph $\Gamma_f$ is a Lagrangian submanifold of the nearly K\"ahler manifold $S^3\times S^3$.
Then the map $f$ is either constant, or satisfies $f(g) = g^{-1}bg$ for some fixed $b \in \SM^2 \subset \SM^3$.
\end{ecor}
\proof The map $g\mapsto f(g)^{-1}$ takes values in $\SM^2\subset \RM^3$. Consider the unit vector field on $\SM^3$ defined by $V:=[u,f^{-1}]$ and the Killing spinor $\Phi:=[\tilde u,1]$.
Theorem \ref{theo} applied to the generalized Killing spinor $\Psi:=[\tilde u,f^{-1}]=V\cdot \Phi$ shows that $V$ is a Hopf vector field. If $V$ is left-invariant then $f$ is constant. If $V$ is right-invariant, $V_g=ag$ for some fixed $a\in \SM^2= \SM^3\cap\RM^3$, which yields $gf(g)^{-1}=ag$ and finally $f(g)=g^{-1}a^{-1}g$.

\qed

\labelsep .5cm

\end{document}